\documentclass[12pt,a4paper]{article}
\usepackage{amssymb}
\usepackage[english]{babel}
\usepackage{amsmath}
\usepackage{amsthm}
\usepackage{amsfonts}
\usepackage[latin1]{inputenc}
\newtheorem{theorem}{Theorem}
\newtheorem{lemma}{Lemma}
\newtheorem{prop}{Proposition}
\newtheorem{cor}{Corollary}
\newtheorem{defi}{Definition}
\newtheorem{remark}{Remark}

\newcommand{\R}{\mathbb{R}}
\newcommand{\N}{\mathbb{N}}

\newcommand{\sfe}{{\mathbb S}^{n-1}}
\newcommand{\ml}{{\mathcal H}^{n}}
\newcommand{\mh}{{\mathcal H}^{n-1}}

\newcommand{\diam}{{\rm diam}}
\newcommand{\nt}{{\rm n.t.}}
\newcommand{\cl}{{\rm cl}}
\newcommand{\interno}{{\rm int}}

\def\<{\langle}
\def\>{\rangle}

\begin{document}

\title{The Minkowski problem for the torsional rigidity}
\author{Andrea Colesanti \& Michele Fimiani}
\date{}

\maketitle
\begin{abstract}
\noindent We prove the existence and uniqueness up to translations of the
solution to a Minkowski type problem for the torsional 
rigidity in the class of open bounded convex subsets of $\R^n$. For the
existence part we apply the variational method introduced by Jerison in:
Adv. Math. 122 (1996), pp. 262--279. Uniqueness follows from the
Brunn--Minkowski inequality for the torsional rigidity and corresponding
equality conditions.  
\end{abstract}

\bigskip


\section{Introduction}\label{intro}

One possible formulation of the Minkowski problem for a set functional $\bf F$ 
is: {\em find (uniquely) a convex set $\Omega$ given 
  the first variation of $\bf F$ as a function of the outer normals to
  $\partial\Omega$}. 

In more precise terms, let $\bf F$ be a real--valued functional defined for every
open bounded convex subset $\Omega$ of $\R^n$ (many Minkowski
  problems are posed in the class of {\it convex
    bodies}, i.e. compact convex subsets of $\R^n$; here we prefer to consider
open sets because the torsional rigidity is usually defined for this type of sets).
Assume that $\bf F$ is positively
homogeneous of some degree $\alpha\ne0$. 
In many examples we see that
associated to $\Omega$ there
exists a non--negative Borel measure $\mu_{{\bf F},\Omega}$ on the unit sphere
$\sfe$ of $\R^n$ such that: ({\em i}) a representation formula holds
\begin{equation}
\label{1.0}
{\bf F}(\Omega)=\frac{1}{\alpha}\int_{\sfe}h(X)\,d\mu_{{\bf F},\Omega}(X)\,,  
\end{equation}
where $h$ is the {\em support function} of (the closure of) $\Omega$;
({\em ii}) a Hadamard's variational formula holds
\begin{equation}
  \label{1.0b}
\left.\frac{d}{dt}{\bf F}(\Omega+t\Omega^\prime)\right|_{t=0^+}=
\int_{\sfe}h^\prime(X)\,d\mu_{{\bf F},\Omega}(X)\,,  
\end{equation}
for every $\Omega^\prime\subset\R^n$ with the same features as
$\Omega$, with corresponding support function $h^\prime$. The last formula
indicates that $\mu_{{\bf F},\Omega}$ is the first 
variation of $\bf F$ when we endow the class of convex sets with the usual
Minkowski addition. If $\bf F$ is translation invariant (which is the
case in all known examples), then a simple
consequence of (\ref{1.0b}) is the vector equation
\begin{equation}
  \label{1.0c}
\int_{\sfe}X\,d\mu_{{\bf F},\Omega}(X)=0\,.  
\end{equation}
Another ``typical'' condition is that:
\begin{equation}
\label{1.0d}
\mbox{$\mu_{{\bf F},\Omega}$ is not supported on any great
sub--sphere of $\sfe$.}
\end{equation}

\bigskip

\noindent{\bf The Minkowski problem for $\bf F$.} {\em Given a non--negative Borel
  measure $\mu$ on $\sfe$ which fulfills conditions (\ref{1.0c}) and
  (\ref{1.0d}), find $\Omega$ such that $\mu=\mu_{{\bf F},\Omega}$.}

\bigskip

When $\bf F$ is the $n$--dimensional volume, this is the classical Minkowski
problem. In this case $\mu_{{\bf F},\Omega}$ is the {\em
  area measure} of the closure of $\Omega$, i.e., for every Borel subset of $\sfe$
$$
\mu_{{\bf F},\Omega}(\eta)=\mh(\nu^{-1}(\eta))
$$
where $\mh$ is the $(n-1)$--dimensional Hausdorff measure and, for
$x\in\partial\Omega$, $\nu(x)$  is the outer unit normal to $\partial\Omega$
at $x$, which is defined for $\mh$--a.e. $x\in\partial\Omega$.
In particular, if $\partial\Omega$ is of class $C^2$ with everywhere positive
Gauss curvature, then 
$$
d\mu_{{\bf F},\Omega}(X)=\frac{1}{\kappa(X)}\,d\mh(X)\,,
$$
where
$\kappa(X)$ is the Gauss curvature at the point $x\in\partial\Omega$ where
$\nu(x)=X$. Hence prescribing the area measure is equivalent to 
assign the Gauss curvature as a function of the outer unit normal. The
classical Minkowski problem is completely solved: we have existence,
and uniqueness up to translations of the solution, and regularity depending on
the smoothness of the datum $\mu$. We refer the 
reader to \cite[Chapter 7]{Schneider} and \cite{Cheng-Yau} for a detailed
presentation of these results.  

In the paper \cite{Jerison2} Jerison proved the validity of (\ref{1.0}) and (\ref{1.0b})
when $\bf F$ is the electrostatic capacity and he established the existence and 
uniqueness up to translations of the solution of the Minkowski problem for
this functional (see also \cite{Caffarelli-Jerison-Lieb}). In a
subsequent paper (\cite{Jerison3}) he gave a new proof of this result (the existence part
only) using a variational approach based on a rather delicate extension of
(\ref{1.0b}). Namely, he proved that for every $f\in C(\sfe)$ (when $\bf F$ is the capacity) 
\begin{equation}
  \label{1.0f}
\left.\frac{d}{dt}{\bf F}(\Omega_t)\right|_{t=0^+}=
\int_{\sfe}f(X)\,d\mu_{{\bf F},\Omega}(X)\,,
\end{equation}
where $\Omega_t$ is determined as follows:
$$
\Omega_t=\{x\in\R^n\,:\,\langle x,X\rangle<
h(X)+tf(X)\,,\;\forall\,X\in\sfe\}
$$
(here we have to assume that the origin belongs to $\Omega$). $\Omega_t$ is
non--empty, open, bounded and convex if $|t|$ is 
sufficiently small; in particular if $f$ is a support function, then the
support function of $\Omega_t$ is $h+tf$, so that (\ref{1.0f})  extends
(\ref{1.0b}). Once (\ref{1.0f}) is established, Jerison solves the variational problem 
\begin{equation}\label{1.0e}
m_{\bf F}=\inf\left\{\int_{\sfe}h_L(X)\,d\mu(X)\,|\,
\mbox{$L$ compact and convex, $\interno(L)\ne\emptyset$, ${\bf F}(\interno(L))\ge 1$}\right\}\,, 
\end{equation}  
(``${\rm int}$'' denotes the interior) and using (\ref{1.0f}) he proves that the Euler--Lagrange equation of
(\ref{1.0e}) is nothing but 
$$
\lambda\mu=\mu_{{\bf F},\Omega}\,,\quad\Omega=\interno(L)\,,
$$
where $\lambda>0$ is a Lagrange multiplier. In this way he obtains the
existence of a solution to the Minkowski problem, 
given by a suitable rescaling of $L$. In
\cite{Jerison2} this method 
is also applied to the Minkowski problem for the transfinite diameter ($n=2$) and for
the first eigenvalue of the Laplace operator with Dirichlet boundary condition. 
It is worth noticing that (\ref{1.0f}) is valid in the case of the volume as
well (see \cite[Lemma 6.5.3]{Schneider}) and the above technique can be
successfully applied to the classic Minkowski problem. 

In this paper, motivated by the work of Jerison, we consider the case when
$\bf F$ is the torsional rigidity. Let us recall that the torsional rigidity
$\tau(\Omega)$ of an open bounded subset $\Omega$ of $\R^n$ (with some 
basic boundary regularity) can be defined as
$$
\tau(\Omega)=\int_\Omega|\nabla u|^2\,dx\,,
$$
where $u$ is the solution of the boundary--value problem
\begin{equation}
\label{1.2}
\left\{
  \begin{array}{ll}
  \Delta u=-2\quad&\mbox{in}\;\Omega\,,\\
  u=0\quad&\mbox{on}\;\partial\Omega\,.
   \end{array}
\right.
\end{equation}
When $\Omega$ is convex, using a result of Dahlberg it can be proved that
$\nabla u$ is defined $\mh$--a.e. on $\partial\Omega$ and $\nabla u\in
L^2(\partial\Omega)$ (see \S \ref{sec2}). In particular the following measure
can be defined on $\sfe$ 
$$
\mu_{\tau,\Omega}(\eta)=\int_{\nu^{-1}(\eta)}|\nabla u|^2\,d\mh\,,
$$
for every Borel subset $\eta$ of $\sfe$. Our first step (see
\S\ref{sec2}) is to prove the validity of 
(\ref{1.0}) and (\ref{1.0b}) for $\bf F=\tau$. Let us remark that
under suitable smoothness assumptions on $\partial\Omega$ these formulas can
be deduced by (a clever use of) the Divergence Theorem (we refer to
\cite{Colesanti} and \cite{Fimiani} for the details). On the other hand the
extension to the general case requires several technical steps. In \S
\ref{sec3} we establish (\ref{1.0f}) for the torsional rigidity. Here the
crucial tool is a geometric lemma proved in \cite{Jerison2}. Collecting these
results and following Jerison's variational approach we obtain an
existence theorem for the Minkowski problem for $\tau$.

\begin{theorem}\label{teo1.1} 
Let $\mu$ be a non--negative Borel measure on $\sfe$;
  assume that 
\begin{equation}\label{1.1}
\int_{\sfe}X\,d\mu(X)=0\,,
\end{equation}
and that the support of $\mu$ is not contained in any great sub--sphere of
$\sfe$. Then there exists an open bounded convex subset $\Omega$ of $\R^n$
such that $\mu=\mu_{\tau,\Omega}$.
\end{theorem}

Let us now come to the issue of uniqueness. In many examples of Minkowski
problems, including the classical one, uniqueness is proved via a {\em
  Brunn--Minkowski type inequality} for the involved functional $\bf F$ and
the corresponding characterization of equality conditions. We say that $\bf F$
satisfies a Brunn-Minkowski inequality with standard 
equality conditions if, for every $\Omega_0$ and $\Omega_1$ open bounded and
convex subsets of $\R^n$ and for every $t\in[0,1]$ we have
\begin{equation}\label{1.BM}
{\bf F}^{1/\alpha}((1-t)\Omega_0+t\Omega_1)\ge
(1-t){\bf F}^{1/\alpha}(\Omega_0)+t{\bf F}^{1/\alpha}(\Omega_1)\,,
\end{equation}
and
\begin{equation}\label{1.BM2}
\mbox{equality holds if and only if $\Omega_1$ is a translate and dilate of $\Omega_0$.}  
\end{equation}
When $\bf F$ is the volume this is the classical Brunn-Minkowski Theorem, see \cite[Chapter
6]{Schneider}. For all the other examples of functionals mentioned before, including the
torsional rigidity, (\ref{1.BM}) and (\ref{1.BM2}) are valid; for the details
we refer the reader to \cite{Colesanti} and the literature quoted
therein. There is a standard argument based on (\ref{1.0b}), (\ref{1.BM}) and
(\ref{1.BM2}) to prove uniqueness in the Minkowski problem for $\bf F$; this
argument can be found in \cite[\S 7.2]{Schneider} or in \cite{Colesanti} in
the case of the volume, but it can be repeated identically in the case of
$\tau$. Hence we have the following completion of Theorem \ref{teo1.1}.

\begin{theorem}\label{teo1.2} In the assumptions of Theorem \ref{teo1.1}, the
  set $\Omega$ is uniquely determined up to a translation.
\end{theorem}

\bigskip

The authors wish to thank professor D. Jerison for his suggestions concerning the proof
of the Hadamard formula for the torsional rigidity.

\section{Some preliminary results}\label{sec2} 

Let $\Omega$ be an open bounded convex subset of
$\R^n$ and consider the solution $u$ of the boundary--value problem
(\ref{1.2}). 
By standard results in the theory of elliptic partial differential equations
(see e.g. \cite{Gilbarg-Trudinger}) $u$ is uniquely determined and belongs to
$C^\infty(\Omega)\cap C(\overline\Omega)$. Note also that, by the strong
maximum principle, $u>0$ in $\Omega$. The function $u$ can be equivalently
defined through a variational problem, indeed it minimizes the functional
$$
\int_\Omega|\nabla v|^2\,dx-4\int_\Omega|v|\,dx\,,
$$
as $v$ ranges in $W^{1,2}_0(\Omega)$. In particular
$u\in W^{1,2}_0(\Omega)$. The torsional rigidity $\tau(\Omega)$ is defined by
$$
\tau(\Omega)=\int_\Omega|\nabla u|^2\,dx\,.
$$

\begin{remark}\label{remark2.1}
{\em If $u$ is the solution of (\ref{1.2}) in $\Omega$ and
    $s\ge0$, then the function 
$$
v(y)=s^2u\left(\frac{y}{s}\right)\,,\quad y\in s\Omega\,,
$$  
is the corresponding solution in $s\Omega=\{y=sx\,|\,x\in\Omega\}$. From this
fact and the definition of torsional rigidity it follows that $\tau$ is
positively homogeneous of order $n+2$:
$$
\tau(s\Omega)=s^{n+2}\tau(\Omega)\,,\quad\forall\,\Omega\,,\,\forall\,s\ge0\,.
$$ 
}
\end{remark}

The convexity of the domain strongly influences the geometry of the solution
$u$. The main result in this direction is the following theorem (see
\cite{Korevaar}, \cite{Kennington}). 

\begin{theorem}\label{teo2.1} 
Let $\Omega$ be an open bounded convex subset of $\R^n$ and let $u$ be the
solution of problem (\ref{1.2}) in $\Omega$. Then $\sqrt u$ is a concave
function in $\Omega$. 
\end{theorem}

Let $M_\Omega=\max_{\overline\Omega}u$; for
every $t\in[0,M_\Omega]$ we define
$$
\Omega_t=\{x\in\Omega\,|\,u(x)> t\}\,.
$$
By Theorem \ref{teo2.1}, $\Omega_t$ is convex for every $t$. Moreover, $\nabla
u(x)=0$ if and only if $u(x)=M_\Omega$ so that
\begin{equation}\label{2.2}
\partial\Omega_t=\{x\in\Omega\,|\,u(x)=t\}\quad\forall\,t\in(0,M_\Omega)\,.  
\end{equation}
Theorem \ref{teo2.1} leads to an $L^\infty$ estimate for the gradient of $u$.

\begin{lemma}\label{lemma2.1} Let $\Omega$ be an open bounded convex subset of
  $\R^n$ and let $u$ be the solution of (\ref{1.2}) in $\Omega$. Then 
$$
|\nabla u(x)|\le\diam(\Omega)\quad\forall\, x\in\Omega\,.
$$  
\end{lemma}
\begin{proof} Let $\bar x\in\Omega$ and $t=u(\bar x)>0$. If $u(\bar x)=M_\Omega$ then $\nabla
  u(\bar x)=0$ and the claim is true. Assume $u(\bar x)<M_\Omega$; this implies that
  $\bar x\in\partial\Omega_t$. The convex set $\Omega_t$ admits a support
  hyperplane $\pi$ at $\bar x$. We may choose an orthogonal coordinate system with
  origin $O$ and coordinates $x_1,\dots,x_n$, in $\R^n$, such
  that $\bar x=O$, $\pi=\{x\in\R^n\,|\,x_n=0\}$ and
  $\Omega_t\subset\{x\in\R^n\,|\,x_n\ge0\}$. By a standard argument based on
  the implicit function theorem, $\partial\Omega_t$ is of class $C^\infty$ so
  that $\pi$ is in fact the tangent hyperplane to $\partial\Omega_t$ at $\bar
  x$. Consequently we have
$$
|\nabla u(\bar x)|=\frac{\partial u}{\partial x_n}(\bar x)\,.
$$
We also have the inclusion $\Omega_t\subset\{x\in\R^n\,|\,x_n\le d\}$,
where $d=\diam(\Omega)$. Let us introduce the function
$$
w(x)=w(x_1,\dots,x_n)=t+x_n(d-x_n)\,,\quad x\in\R^n\,.
$$
Note that $\Delta w(x)=-2$ for every $x$ in $\R^n$ and $w(x)\ge t$ for 
$x\in\{x=(x_1,\dots,x_n)\in\R^n\,|\,0\le x_n\le d\}$. In particular $w\ge u$
on $\partial\Omega_t$ and, by the Comparison Principle, 
$$
w(x)\ge u(x)\quad\forall\,x\in\Omega_t\,.
$$
Finally, as $u(\bar x)=w(\bar x)$,
$$
\frac{\partial u}{\partial x_n}(\bar x)\le
\frac{\partial w}{\partial x_n}(\bar x)=d\,.
$$   
\end{proof}

Next we investigate the boundary behavior of $\nabla u$. 
We will use the notion of {\em non--tangential limit} of a function at a
boundary point of a domain, which we briefly recall (for further details we
refer the reader to \cite{Kenig}). Let $\Omega$ be an open bounded subset of
$\R^n$ and let $\bar x\in\partial\Omega$. For $\alpha>0$ we define the {\em
  non--tangential cone}
$$
\Gamma_\alpha(\bar x)=\{x\in\Omega\,|\,|x-\bar x|\le(1+\alpha)\,{\rm
  dist}(x,\partial\Omega)\}\,.
$$
We say that a sequence of points $x_i\in\Omega$, $i\in\N$, converges
non--tangentially to $\bar x\in\partial\Omega$ if, for some $\alpha>0$, 
$$
\lim_{i\to\infty}x_i=\bar x\quad{\rm and}\quad x_i\in\Gamma_\alpha(\bar
x)\quad\forall\, i\in\N\,.
$$
Moreover, we say that a function $w$ defined in $\Omega$ admits {\em non--tangential
  limit} $L$ at $\bar x\in\partial\Omega$ if 
$$
\lim_{x\to\bar x,\,x\in\Gamma_\alpha(\bar x)}w(x)=L
$$
and such a limit does not depend on $\alpha$. In this case we write
$$
\lim_{x\to\bar x\,\nt}w(x)=L\,.
$$
By a {\em Lipschitz set} we mean a set which can be expressed locally, after a
suitable choice of the coordinate system, as the
epigraph of a Lipschitz function of $(n-1)$ variables. Note that an open bounded convex set is a
Lipschitz set. We quote an important result by Dahlberg (see \cite{Dahlberg}). 

\begin{theorem}\label{teo2.2} Let $\Omega$ be a Lipschitz open subset of
  $\R^n$  and let $w$ be harmonic and bounded by below in $\Omega$.
  Then for $\mh$--a.e. point $x\in\partial\Omega$, $w$ has a finite
  non--tangential limit at $x$.  
\end{theorem}

\begin{prop}\label{prop2.1} Let $\Omega$ be an open bounded convex set of $\R^n$ and let $u$
  be the solution of (\ref{1.2}) in $\Omega$. Then for
  $\mh$--a.e. $x\in\partial\Omega$, $\nabla u$ has finite non--tangential limit at
  $x$, i.e. each component of $\nabla u$ has finite
  non--tangential limit at $x$.   
\end{prop}

\begin{proof} Let us fix $\bar t\in(0,M_\Omega)$ and let $x^\prime\in\Omega$
  be such that $u(x^\prime)=M_\Omega$; note that $\Omega_{\bar t}$ is an open
  set, whose closure is contained in $\Omega$, and $x^\prime\in\Omega_{\bar
  t}$. Hence, there exists $\rho>0$ such that 
$$
\Omega_{\bar t}\supset B_\rho=\;\mbox{open ball of radius $\rho$
  centered at $x^\prime$.}
$$   
For every $t\in(0,\bar t)$, $\Omega_t\supset\Omega_{\bar t}\supset B_\rho$; since
$\Omega_t$ is convex, is also star--shaped with respect to every
point of $B_\rho$. We consider the set
$$
S=\Omega\setminus{\rm cl}(\Omega_{\bar t})=\{x\in\Omega\,|\,u(x)<\bar t\}
$$
(where ``${\rm cl}$'' denotes the closure of a set). Let $y$ be a point in
$B_\rho$; as the super--level set are star--shaped with respect to $y$ we have
that the function
$$
\tilde w(x)=\langle\nabla u(x),y-x\rangle
$$
is non--negative in $S$. Moreover $\Delta\tilde w=-4$ in $S$. As a consequence,
the function
$$
w(x)=\tilde w(x)-\frac{2}{n}|x|^2\,,\quad x\in S\,,
$$
is harmonic and bounded by below in $S$. As $\partial
S=\partial\Omega\cup\partial\Omega_{\bar t}$, applying Theorem \ref{teo2.2} we
deduce that $\mh$--a.e. on $\partial\Omega$, $w$ has finite non--tangential limit and
this implies that the same is true for the function $\tilde w$. Now, we observe that the
point $y$ in the definition of $\tilde w$ can be chosen arbitrarily in $B_\rho$;
if we apply the above argument for $y=x^\prime$ and
$y=x^\prime+\dfrac{\rho}{2}e_1$ respectively, where $\{e_1,\dots,e_n\}$ is the canonical
basis in $\R^n$, we obtain that the function
$$
\langle\nabla u(x),e_1\rangle=\frac{\partial u}{\partial x_1}
$$
admits finite non--tangential limit for $\mh$--a.e. point of
$\partial\Omega$. The same can be done for every component of $\nabla u$ and
this completes the proof.
\end{proof}

According to the above result, if $\Omega$ is an open bounded convex subset of
$\R^n$ and $u$ is the solution of problem (\ref{1.2}) in $\Omega$, then
$\nabla u$ is defined at $\mh$--a.e. point of $\partial\Omega$ and, by Lemma
\ref{lemma2.1}, $\nabla u\in L^\infty(\partial\Omega)$. 
As we already mentioned $\Omega$ is a  Lipschitz set 
and then $\partial\Omega$ is differentiable 
$\mh$--a.e.; consequently for $\mh$--a.e. $x\in\partial\Omega$ the outer
unit normal $\nu(x)\in\sfe$ is defined. The map $\nu$ is called the {\em Gauss
  map} and it is usually defined for the closure $\cl(\Omega)$ instead of
$\Omega$ itself. 
Note that $\nu^{-1}$ maps Borel subsets of $\sfe$ into
$\mh$--measurable subsets of $\partial\Omega$ (see \cite[Lemma 2.2.11]{Schneider}).  
We are now in position to define the main ingredient of this
paper, i.e. the measure $\mu_{\tau,\Omega}$, which corresponds to the notion of 
{\em area measure} when the volume is replaced by the torsional rigidity.

\begin{defi}\label{def2.1} Let $\Omega$ be an open bounded convex subset of
  $\R^n$ and let $u$ be the solution of problem (\ref{1.2}) in $\Omega$. Let 
  $\nu$ be the Gauss map of $\Omega$. For every Borel subset $\eta$ of $\sfe$
  we set
$$
\mu_{\tau,\Omega}(\eta)=\int_{\nu^{-1}(\eta)}|\nabla u(x)|^2\,d\mh(x)\,.
$$
Hence $\mu_{\tau,\Omega}$ is a non--negative Borel measure on $\sfe$. 
\end{defi}

\section{The Hadamard formula for $\tau$}\label{sec3}

Before we state the main result of this section, let us briefly recall the
notion of set addition (or {\em Minkowski addition}) and some other facts from
Convex Geometry. Let $A$ and 
$B$ two subsets of $\R^n$; their sum is defined as 
$$
A+B=\{x+y\,|\,x\in A\,,\,y\in B\}\,.
$$
Note that if $A$ and $B$ are open (resp. closed, bounded, convex), then $A+B$
is open (resp. closed, bounded, convex). Let $K\subset\R^n$ be a {\em convex
  body}, i.e. a compact convex set. The support function $h_K$ of $K$ is defined as
$$
h_K\,:\,\sfe\to\R\,,\quad h_K(X)=\sup_{x\in K}\langle X,x\rangle\,.
$$
Roughly speaking, $h_K(X)$ is the signed distance from the origin of the
supporting hyperplane to $K$ having $X$ as outer normal. In particular, if the
origin belongs to $K$ then $h_K$ is non--negative and $h_K(X)\le{\rm diam}(K)$
for every $X\in\sfe$. We refer to
\cite{Schneider} for further properties of the support function. 

\begin{theorem}\label{teo3.1} 
Let $\Omega$ be an open bounded convex subset of $\R^n$, $u$
  be the solution of (\ref{1.2}) in $\Omega$ and let $h$
  be the support function of ${\rm cl}(\Omega)$. Then 
\begin{equation}\label{3.1}
\tau(\Omega)=\frac{1}{n+2}
\int_{\partial\Omega}h(\nu(x))|\nabla u(x)|^2\,d\mh(x)\,,
\end{equation}
where $\nu$ is the Gauss map of $\cl(\Omega)$. 
Moreover, let $\Omega^\prime$ be another open bounded convex subset of $\R^n$
and let $h^\prime$ be the support function of ${\rm cl}(\Omega^\prime)$. Then 
\begin{equation}\label{3.2}
\lim_{s\to0^+}\frac{\tau(\Omega+s\Omega^\prime)-\tau(\Omega)}{s}=
\int_{\partial\Omega}h^\prime(\nu(x))|\nabla u(x)|^2\,d\mh(x)\,.
\end{equation}
\end{theorem}

By the above result and Definition \ref{def2.1} we immediately get the validity of formulas
(\ref{1.0}) and (\ref{1.0b}) for $\tau$. 

\begin{cor}\label{cor3.0} In the assumptions and notations of Theorem \ref{teo3.1} we have
\begin{equation}\label{3.1bis}
\tau(\Omega)=\frac{1}{n+2}
\int_{\sfe}h(X)\,d\mu_{\tau,\Omega}(X)\,,
\end{equation}  
and
\begin{equation}\label{3.2bis}
\lim_{s\to0^+}\frac{\tau(\Omega+s\Omega^\prime)-\tau(\Omega)}{s}=
\int_{\sfe}h^\prime(X)\,d\mu_{\tau,\Omega}(X)\,.
\end{equation}  
\end{cor}

Some comments are in order. Under the assumption that the boundaries of
$\Omega$ and $\Omega^\prime$ are 
of class $C^2$, the theorem was stated in \cite[Proposition 18]{Colesanti},
with a proof of (\ref{3.1}) and a sketch of the proof of (\ref{3.2}). A
detailed proof of the latter equality can be found in \cite{Fimiani}.
In order to remove the regularity assumption
we will combine several ingredients: ({\em i}) the validity of the
equality for sets with smooth boundary; ({\em ii}) the density, with respect
to the Hausdorff metric, of convex bodies with smooth boundary in the class of
all convex bodies; ({\em iii}) the continuity, in both arguments, of the set functional 
$$
(\Omega,\Omega')\longrightarrow  
\int_{\partial\Omega}h^\prime(\nu_\Omega(x))|\nabla u(x)|^2\,d\mh(x)
$$
(with respect to the Hausdorff metric). The main effort will be required for
part ({\em  iii}). The proof of Theorem \ref{teo3.1} is preceded by some preparatory
lemmas.  In the sequel we use the notion of Hausdorff distance and Hausdorff
metric, for which we refer to \cite{Schneider}. 

\begin{lemma}\label{lemma3.1} Let $K\subset\R^n$ be a convex body with
  non--empty interior and let $K_i$, $i\in\N$, be a sequence of convex bodies
  converging to $K$ in the Hausdorff metric. Then there exists a sequence
  $\alpha_i$, $i\in\N$, such that
$$
\lim_{i\to\infty}\alpha_i=1\,,\quad
\alpha_i K_i\subset{\rm int}(K)\,,\;\forall\,i\in\N\,,
$$
and $\alpha_i K_i$ converges to $K$ in the Hausdorff metric. 
\end{lemma}

\begin{proof} We recall that, in the set of convex bodies, the convergence
  with respect to the Hausdorff metric is equivalent to uniform convergence of
  support functions on $\sfe$ (see \cite{Schneider}). It is not restrictive to
  assume that for some $\rho>0$
$$
B_\rho\subset K\,,\,K_i\quad\forall i\in\N\,,
$$
where $B_\rho$ is the ball centered at the origin with radius $\rho$. This
implies $h(X)\ge\rho$ and $h_i(X)\ge\rho$ for every $X\in\sfe$ and for every
$i\in\N$. Similarly, as ${\rm diam}(K_i)\to{\rm diam}(K)$, the functions $h$
and $h_i$, $i\in\N$, are uniformly bounded from above. We set
$$
\alpha_i=\min_{X\in\sfe}\frac{h(X)}{h_i(X)}-\frac{1}{i}\,,\quad i\in\N\,.
$$
By the uniform convergence, $\alpha_i>0$ (at least definitively) and 
$\alpha_i\to 1$ as $i$ tends to infinity. We also have that
$$
\alpha_i\,h_i(X)<h(X)\quad\forall\,X\in\sfe\,,\;\forall\,i\in\N\,,
$$
which implies $\alpha_iK_i\subset{\rm int}(K)$. Finally $\alpha_i\,h_i$
converges uniformly to $h$ in $\sfe$, i.e. $\alpha_iK_i$ converges to $K$ in
the Hausdorff metric. 
\end{proof}

Let $K\subset\R^n$ be a convex body such that the origin is an interior
point of $K$. For $\theta\in\sfe$ we set
$$
\rho_K(\theta)=\sup\{\rho\ge0\,|\,\rho\theta\in K\}\,.
$$
This is the {\em radial function} of $K$. The corresponding {\em radial map} is
defined by
$$
r_K\,:\,\sfe\to\partial K\,,\quad r_K(\theta)=\rho_K(\theta)\,\theta\,.
$$
In other words, $r_K(\theta)$ is the (unique) intersection of $\partial\Omega$ with the
ray from the origin parallel to $\theta$. 
For the reader's convenience, throughout
this paper the variable of radial functions will be denoted by $\theta$ while
the variable of support functions will be denoted by $X$, though they are both
defined on $\sfe$. 

Let $\Omega={\rm int}(K)$; if
$f\,:\,\partial\Omega\to\R$ is $\mh$--integrable we have the following
formula for the change of variable given by the radial map:
\begin{equation}
  \label{3.3}
\int_{\partial\Omega}f(x)\,d\mh(x)=
\int_{\sfe}f(r_K(\theta))\frac{\rho_K^n(\theta)}{h_K(\nu_K(\rho_K(\theta)))}\,d\mh(\theta)\,.
\end{equation}
Let $\Omega_i$, $i\in\N$ be a sequence of open bounded convex sets
such that $K_i={\rm cl}(\Omega_i)$ converges to $K$ in the Hausdorff metric as
$i\to\infty$; then $\rho_{K_i}$ converges uniformly to $\rho_K$ on $\sfe$ and
\begin{equation}\label{3.4}
R_i(\theta):=\frac{\rho_{K_i}^n(\theta)}{h_{K_i}(\nu_{K_i}(\rho_{K_i}(\theta)))}
\longrightarrow
\frac{\rho_K^n(\theta)}{h_K(\nu_K(\rho_K(\theta)))}\quad\mbox{for $\mh$--a.e. $\theta\in\sfe$.} 
\end{equation}
Moreover the functions $R_i$ are uniformly bounded above and below by positive
constants depending on the {\em inner radius} of $\Omega$ (i.e. the radius
of the largest ball contained in $\Omega$) and the diameter of $\Omega$. 

\begin{lemma}\label{lemma3.2} Let $\Omega$, $\Omega_i$, $i\in\N$, be open
  bounded convex subsets of $\R^n$ and assume that the sequence of convex
  bodies $K_i={\rm cl}(\Omega_i)$, $i\in\N$, converges to $K={\rm cl}(\Omega)$
  in the Hausdorff metric. Let 
$$
f\,:\,\partial\Omega\rightarrow\R\,,\quad
f_i\,:\,\partial\Omega_i\rightarrow\R\,,\quad i\in\N\,,
$$
be $\mh$--measurable functions such that:
\begin{itemize}
\item[i.] there exists $C>0$ for which
$$
\|f\|_{L^\infty(\partial\Omega)}\le C\,,\quad
\|f_i\|_{L^\infty(\partial\Omega_i)}\le C\,,\quad\forall\,i\in\N\,;
$$
\item[ii.] for $\mh$--a.e. $x\in\partial\Omega$, if $x_i\in\partial\Omega_i$,
  $i\in\N$, is such that $x_i\to x$ non--tangentially and $f_i$ is defined in $x_i$, then
$$
\lim_{i\to\infty}f(x_i)=f(x)\,.
$$
\end{itemize}
Under these conditions we have
\begin{equation}\label{3.5}
\lim_{i\to\infty}\int_{\partial\Omega_i}f_i(x)\,d\mh(x)=
\int_{\partial\Omega}f(x)\,d\mh(x)\,.
\end{equation}
\end{lemma}

\begin{proof} We may assume that the origin is an interior point of $\Omega$
  and of $\Omega_i$, for every $i\in\N$. Let $\rho$ and $r$ be the radial function
  and the radial map respectively, of $K={\rm cl}(\Omega)$, and, for $i\in\N$, let $\rho_i$
  and $r_i$ be the corresponding objects associated to $K_i={\rm
  cl}(\Omega_i)$. For $i\in\N$ let
$$
A_i=\{\theta\in\sfe\,|\,\mbox{$f_i(r_i(\theta))$ is defined}\}\quad
\mbox{and}\quad A=\bigcap_{i=1}^\infty A_i\,.
$$
Denote by $A^\prime$ the set of those points $\theta\in A$ such that
assumption {\em ii.} of the theorem holds at $r(\theta)$; we have that
$\mh(\sfe\setminus A^\prime)=0$. Note that for every $\theta\in\sfe$ the
sequence $r_i(\theta)$ converges to $r(\theta)$ non--tangentially, since all
these points lie on the same ray from the origin, and the origin is in the
interior of $\Omega$. Hence
\begin{equation}\label{3.6}
\lim f_i(r_i(\theta))=f(r(\theta))\quad\forall\,\theta\in A^\prime\,.
\end{equation}
To conclude the proof, apply the change of variable formula (\ref{3.3}) to
both sides of (\ref{3.5}); then, by (\ref{3.4}), (\ref{3.6}) and assumption
{\em i.}, we may apply the Dominated Convergence Theorem and obtain equality
(\ref{3.5}). 
\end{proof}

\begin{remark}\label{remark3.1}{\em Let $\Omega$, $\Omega_i$, $i\in\N$, be open
  bounded convex subsets of $\R^n$ and assume that the sequence of convex
  bodies $K_i={\rm cl}(\Omega_i)$, $i\in\N$, converges to $K={\rm cl}(\Omega)$
  in the Hausdorff metric. Let $\nu$ and $\nu_i$ denote the Gauss map of $K$
  and of $K_i$, $i\in\N$, respectively. Let 
  $x\in\partial\Omega$ be a point where $\partial\Omega$ is differentiable,
  i.e. $\nu(x)$ is defined and let $x_i\in\partial\Omega_i$, $i\in\N$, be such
  that
$$
\lim_{i\to\infty}x_i=x\,,\quad\mbox{$\partial\Omega_i$ is differentiable at
  $x_i$, for every $i$.}
$$
Then
$$
\lim_{i\to\infty}\nu_i(x_i)=\nu(x)\,.
$$
}\end{remark}

\begin{lemma}\label{lemma3.3} Let $\Omega$, $\Omega_i$, $i\in\N$, be open
  bounded convex subsets of $\R^n$ and assume that the sequence of convex
  bodies $K_i={\rm cl}(\Omega_i)$, $i\in\N$, converges to $K={\rm cl}(\Omega)$
  in the Hausdorff metric. Assume moreover that ${\rm
  cl}(\Omega_i)\subset\Omega$ for every $i\in\N$. Let $u$ be the solution of
  problem (\ref{1.2}) in $\Omega$. Let $x\in\partial\Omega$ be such that
  $\nabla u$ has finite non--tangential limit at $x$ and $\partial\Omega$ is
  differentiable at $x$. Let $x_i\in\partial\Omega_i$, $i\in\N$, be such that
  $x_i$ converges non--tangentially to $x$ and $\partial\Omega_i$ is
  differentiable at $x_i$ for every $i\in\N$. Define 
$$
\nabla_T u(x_i)=\nabla u(x_i)-\langle\nabla
u(x_i),\nu_i(x_i)\rangle\nu_i(x_i)\,,\quad i\in\N\,,
$$  
where $\nu_i$ is the Gauss map of $\cl(\Omega_i)$. Then
$$
\lim_{i\to\infty}\nabla_Tu(x_i)=0\,.
$$
\end{lemma}
 
\begin{proof} If
$$
\lim_{y\to x\,{\rm n.t.}}\nabla u(y)=0\,,
$$
then there is nothing to prove. Then we assume that the above limit is not the
null vector. For every $i\in\N$ we set $\epsilon_i=u(x_i)>0$. Let $\bar\nu_i$
denote the Gauss map of $\cl(\Omega_{\epsilon_i})$, for $i\in\N$; then
$$
-\frac{1}{|\nabla u(x_i)|}\,\nabla
u(x_i)=\bar\nu_i(x_i)\quad\forall\,i\in\N\,.
$$
By the assumptions of the present lemma and by Remark \ref{remark3.1} we have
$$
\lim_{i\to\infty}\nu_i(x_i)=\nu(x)\,,\quad
\lim_{i\to\infty}\bar\nu_i(x_i)=\nu(x)\,.
$$
The claim follows from the definition of $\nabla_T u$. 
\end{proof}

Let $\Omega$, $\Omega^\prime$ be open bounded convex subsets of $\R^n$; let
$u$ be the solution of (\ref{1.2}) in $\Omega$, let $h^\prime$ be the
support function of ${\rm cl}(\Omega^\prime)$ and let $\nu$ be the Gauss map of
$\cl(\Omega)$. We define the functional 
$$
\tau_1(\Omega,\Omega^\prime)=\int_{\partial\Omega}h^\prime(\nu(x))|\nabla
u(x)|^2\,d\mh(x)\,. 
$$
According to Theorem \ref{teo3.1} and by analogy with the case of the volume,
$\tau_1$ can be seen as the {\em mixed torsional rigidity} of $\Omega$ and
$\Omega^\prime$.  

\begin{remark}\label{remark3.2}{\em Let $\Omega$ and $\Omega^\prime$ be as
    above and let $s>0$. If $\nu$ and $\nu_s$ denote the Gauss maps of
    $\cl(\Omega)$ and $\cl(s\Omega)$ respectively, then $\nu_s(sx)=\nu(x)$ for every
    $x\in\partial\Omega$. This fact, together with Remark \ref{remark2.1}
    shows that
\begin{equation}\label{3.7}
\tau_1(s\Omega,\Omega^\prime)=s^{n+1}\tau_1(\Omega,\Omega^\prime)\,.
\end{equation}  
The functional $\tau_1$ is homogeneous with respect to its second
variable also; indeed, if $h_s$ is the support function of $\cl(s\Omega^\prime)$,
then $h_s(X)=sh(X)$ for every $X\in\sfe$. Hence
\begin{equation}\label{3.7b}
\tau_1(\Omega,s\Omega^\prime)=s\tau_1(\Omega,\Omega^\prime)\,.
\end{equation} 
}
\end{remark}

\begin{theorem}\label{teo3.2} The functional $\tau_1$ is continuous
  with respect to the Hausdorff metric, i.e. let $\Omega$
  and $\Omega^\prime$ be as above and let $\Omega_i$, $\Omega^\prime_i$,
  $i\in\N$ be two sequences of open bounded convex sets such that
  $\cl(\Omega_i)$ and  $\cl(\Omega^\prime_i)$ converge in the Hausdorff metric
  to $\cl(\Omega)$ and $\cl(\Omega^\prime)$ respectively, as $i$ tends to infinity. Then
$$
\lim_{i\to\infty}
\tau_1(\Omega_i,\Omega_i^\prime)=
\tau_1(\Omega,\Omega^\prime)\,.
$$
\end{theorem}

\begin{proof} Let $h^\prime$ and $h^\prime_i$, $i\in\N$ be the support functions of
  $\cl(\Omega^\prime)$ and $\cl(\Omega_i^\prime)$ respectively. By the
  assumptions we have
\begin{equation}\label{3.8}
h_i\longrightarrow\,h\quad\mbox{uniformly on $\sfe$}\,.
\end{equation}
A simple argument
  based on Lemma \ref{3.1} and equality (\ref{3.7}) shows that we may assume
  without loss of generality that $\cl(\Omega_i)\subset\Omega$ for every
  $i\in\N$. Let $u_i$ be the solution of problem (\ref{1.2}) in $\Omega_i$ and
  let $\nu_i$ be the Gauss map of $\cl(\Omega_i)$. Our goal is to prove that
  \begin{equation}\label{3.9}
\lim_{i\to\infty}
\int_{\partial\Omega_i}h^\prime_i(\nu_i(x))|\nabla u_i(x)|^2\,d\mh(x)=
\int_{\partial\Omega}h^\prime(\nu(x))|\nabla u(x)|^2\,d\mh(x)\,. 
  \end{equation}
We have that
\begin{eqnarray}\label{3.9b}
&&\left|\int_{\partial\Omega_i}h^\prime_i(\nu_i(x))|\nabla u_i(x)|^2\,d\mh(x)-
\int_{\partial\Omega}h^\prime(\nu(x))|\nabla u(x)|^2\,d\mh(x)\right|
\le\nonumber\\
&&\int_{\partial\Omega_i}|h^\prime_i(\nu_i(x))|\left||\nabla u_i(x)|^2-|\nabla u(x)|^2\right|\,d\mh(x)+\\
&&+\left|\int_{\partial\Omega_i}h^\prime_i(\nu_i(x))|\nabla u(x)|^2\,d\mh(x)-
\int_{\partial\Omega}h^\prime(\nu(x))|\nabla u(x)|^2\,d\mh(x)\right|\,.\nonumber
\end{eqnarray}
We first deal with the second summand of the right hand--side of the above inequality. Let 
$$
f_i(x)=h^\prime_i(\nu_i(x))|\nabla u(x)|^2\,,\quad i\in\N\,,\quad
f(x)=h^\prime(\nu(x))|\nabla u(x)|^2\,.
$$
Note that there exists a constant $C>0$  such that
$$
\diam(\Omega)\,,\,
\diam(\Omega^\prime)\,,\,
\diam(\Omega_i)\,,\,
\diam(\Omega^\prime_i)\le C\quad\forall\,i\in\N\,.
$$
Hence, by Lemma \ref{lemma2.1},
$$
\|f\|_{L^\infty(\partial\Omega)}\,,\,
\|f_i\|_{L^\infty(\partial\Omega_i)}\le C^\prime\quad\forall\,i\in\N\,,
$$
for some $C^\prime>0$. 
Let $x\in\partial\Omega$ be such that $\nabla u$ admits non--tangential limit
at $x$, and let 
$x_i\in\partial\Omega_i$, $i\in\N$, be such that $\nu(x)$ and $\nu_i(x_i)$ are
defined for every $i$, and assume that $x_i$ tends non--tangentially to
$x$. Then, $\nu_i(x_i)\to\nu(x)$ (see Remark \ref{remark3.1}) and, by
(\ref{3.8}), $h^\prime_i(\nu_i(x_i))\to h^\prime(\nu(x))$ and
$|\nabla u(x_i)|^2\to|\nabla u(x)|^2$. Applying
Lemma \ref{lemma3.2} and Proposition \ref{prop2.1} we get
\begin{equation}\label{3.10}
\lim_{i\to\infty}\int_{\partial\Omega_i}h^\prime_i(\nu_i(x))|\nabla u(x)|^2\,d\mh(x)=
\int_{\partial\Omega}h^\prime(\nu(x))|\nabla u(x)|^2\,d\mh(x)\,.
\end{equation}
Next we prove that
\begin{equation}\label{3.11}
\lim_{i\to\infty}\int_{\partial\Omega_i}|h^\prime_i(\nu_i(x))|\left||\nabla u_i(x)|^2-|\nabla
  u(x)|^2\right|\,d\mh(x)=0\,. 
\end{equation}
This fact, together with (\ref{3.10}) and (\ref{3.9b}) leads to (\ref{3.9})
and then to the conclusion of the proof. Note that
\begin{eqnarray*}
\int_{\partial\Omega_i}|h^\prime_i(\nu_i(x))|\left||\nabla u_i(x)|^2-|\nabla
  u(x)|^2\right|\,d\mh(x)
&\le& C\,\int_{\partial\Omega_i}\left||\nabla u_i(x)|^2-|\nabla
  u(x)|^2\right|\,d\mh(x)\\
&\le&
C_1\,\int_{\partial\Omega_i}|\nabla(u_i(x)-u(x))|\,d\mh(x)\\
&\le& C_2\left(\int_{\partial\Omega_i}|\nabla(u_i(x)-u(x))|^2\,d\mh(x)\right)^{\frac{1}{2}}\,,
\end{eqnarray*}
where $C_1$, $C_2>0$ are independent of $i$. Here we used the fact that, since
the diameters of $\cl(\Omega_i)$ are uniformly bounded, the same holds for
$\mh(\partial\Omega_i)$. In the remaining part of the proof
we will show that 
\begin{equation}
  \label{3.12}
\lim_{i\to\infty}\int_{\partial\Omega_i}|\nabla(u_i(x)-u(x))|^2\,d\mh(x)=0\,. 
\end{equation}
For $i\in\N$ and $\epsilon>0$ sufficiently small, consider the set 
$$
\Omega^\epsilon_i=\{x\in\Omega_i\,|\,u_i(x)>\epsilon\}\,.
$$
$\Omega^\epsilon_i$ is an open set with boundary of class $C^\infty$;
moreover, for fixed $i\in\N$, $\cl(\Omega^\epsilon_i)$ converges to $\cl(\Omega_i)$ in the Hausdorff
metric as $\epsilon\to0^+$. The function $u_i-u$ is harmonic in
$\Omega^\epsilon_i$ and $u_i\equiv\epsilon$ on $\partial\Omega^\epsilon_i$. We
apply Corollary 2.1.14 of \cite{Kenig} to get
\begin{equation}\label{3.14}
\int_{\partial\Omega^\epsilon_i}|\nabla(u_i(x)-u(x))|^2\,d\mh(x)\le
C_3\int_{\partial\Omega^\epsilon_i}|\nabla_{T_i^\epsilon} u(x)|^2\,d\mh(x)\,, 
\end{equation}
where $C_3$ is a constant depending only on the inner radius and the diameter of
$\Omega$ and $\nabla_{T_i^\epsilon} u$ is the tangential component of $\nabla u$ to
$\partial\Omega^\epsilon_i$:
$$
\nabla_{T_i^\epsilon} u=\nabla u-\langle\nabla u,\nu^\epsilon_i\rangle\,\nu^\epsilon_i
$$
(here $\nu^\epsilon_i$ is the Gauss map of
$\cl(\Omega^\epsilon_i)$). As above, we may apply Lemma \ref{lemma2.1},
Proposition \ref{prop2.1}, Remark \ref{remark3.1} and Lemma \ref{lemma3.2} to
deduce (for a fixed $i$)
\begin{eqnarray*}
&&\lim_{\epsilon\to0^+}\int_{\partial\Omega^\epsilon_i}|\nabla(u_i(x)-u(x))|^2\,d\mh(x)=
\int_{\partial\Omega_i}|\nabla(u_i(x)-u(x))|^2\,d\mh(x)\,,\\
&&\lim_{\epsilon\to0^+}\int_{\partial\Omega^\epsilon_i}|\nabla_{T_i^\epsilon}
u(x)|^2\,d\mh(x)=\int_{\partial\Omega_i}|\nabla_{T_i} u(x)|^2\,d\mh(x)\,,
\end{eqnarray*}
where
$$
\nabla_{T_i} u=\nabla u-\langle\nabla u,\nu_i\rangle\,\nu_i
$$
is defined $\mh$--a.e. on $\partial\Omega_i$. Hence
\begin{equation}\label{3.15}
\int_{\partial\Omega_i}|\nabla(u_i(x)-u(x))|^2\,d\mh(x)\le
C_3\int_{\partial\Omega_i}|\nabla_{T_i} u(x)|^2\,d\mh(x)\,.
\end{equation}
By Lemma \ref{lemma3.3}, Lemma \ref{lemma2.1} and Lemma \ref{lemma3.2} we get
\begin{equation}\label{3.16}
\lim_{i\to\infty}\int_{\partial\Omega_i}|\nabla_T u(x)|^2\,d\mh(x)=0\,. 
\end{equation}
Equality (\ref{3.12}) follows from (\ref{3.15}) and (\ref{3.16}); the proof is
complete. 
\end{proof}

\begin{proof}[Proof of Theorem \ref{teo3.1}] Let us start with formula
  (\ref{3.1}). We recall that the equality is true under the assumption that
  the boundary of the domain is of class $C^2$ (see \cite{Colesanti}). Let
  $\epsilon>0$ be smaller than $\max_{\bar\Omega}u$ and consider the
  super--level set 
$$
\Omega_\epsilon=\{x\in\Omega\,|\,u(x)>\epsilon\}\,.
$$
We know that $\Omega_\epsilon$ is convex, $\partial\Omega_\epsilon$ is of
class $C^\infty$ and $\cl(\Omega_\epsilon)\to\cl(\Omega)$ in the Hausdorff
metric as $\epsilon\to0^+$. Notice that the function $u_\epsilon=u-\epsilon$
is the solution of problem (\ref{1.2}) in $\Omega_\epsilon$; consequently
$$
\tau(\Omega_\epsilon)=\int_{\Omega_\epsilon}|\nabla u|^2\,dx\,.
$$    
As $\|\nabla\|_{L^\infty(\Omega)}<\infty$ and
$\ml(\Omega\setminus\Omega_\epsilon)\to0$ when $\epsilon\to0^+$ (where $\ml$
denotes the Lebesgue measure in $\R^n$), we obtain
$$
\lim_{\epsilon\to0^+}\tau(\Omega_\epsilon)=\tau(\Omega)\,.
$$
Moreover we have
$$
\tau(\Omega_\epsilon)=\frac{1}{n+2}
\int_{\partial\Omega_\epsilon}h_\epsilon(\nu_\epsilon(x))|\nabla
u(x)|^2\,d\mh(x)=\frac{1}{n+2}
\tau_1(\Omega_\epsilon,\Omega_\epsilon)\,,
$$
where $h_\epsilon$ and $\nu_\epsilon$ are the support function and the Gauss
map of $\Omega_\epsilon$ respectively. Passing to the limit for
$\epsilon\to0^+$ and using Theorem \ref{teo3.2} we get
$$
\tau(\Omega)=\frac{1}{n+2}\tau_1(\Omega,\Omega)
$$
i.e. (\ref{3.1}). 

Next we prove (\ref{3.2}). Let $\Omega_i$,
$\Omega^\prime_i$, $i\in\N$ be two sequences of open bounded convex sets, with
boundaries of class $C^2$, such that $\cl(\Omega_i)\to\cl(\Omega)$ and
$\cl(\Omega^\prime_i)\to\cl(\Omega^\prime)$ as $i$ tends to infinity, in the
Hausdorff metric. For $t\ge0$ let
$$
\Omega_t=\Omega+t\Omega^\prime\,,\quad
\Omega_{i,t}=\Omega_i+t\Omega^\prime_i\,,\quad i\in\N\,.
$$
We know that (see \cite[equality (30)]{Colesanti}), for every
$i\in\N$ and every $t>0$ we have
\begin{equation}\label{3.17}
\frac{d}{dt}\tau(\Omega_{i,t})=
\int_{\partial\Omega_{i,t}}h^\prime_i(\nu_{i,t}(x))|\nabla
u_{i,t}(x)|^2\,d\mh(x)
=\tau_1(\Omega_{i,t},\Omega^\prime_i)\,,
\end{equation}
where: $h^\prime_i$ is the support function of $\cl(\Omega^\prime_i)$, $\nu_{i,t}$
is the Gauss map of $\cl(\Omega_{i,t})$, $u_{i,t}$ is the solution of problem
(\ref{1.2}) in $\Omega_{i,t}$.
By (\ref{3.17}) we have
\begin{equation}\label{3.18}
\tau(\Omega_{i,t})-\tau(\Omega_i)=\int_0^t\tau_1(\Omega_{i,s},\Omega^\prime_i)\,ds\,,\quad t\ge0\,.  
\end{equation}
Moreover, for every $t\ge0$, $\cl(\Omega_{i,t})$ converges to $\cl(\Omega_t)$ as
$i$ tends to infinity. Consequently
\begin{equation}\label{3.19}
\lim_{i\to\infty}\tau(\Omega_{i,t})=\tau(\Omega_t)\quad\mbox{and}
\quad\lim_{i\to\infty}\tau_1(\Omega_{i,t},\Omega^\prime_i)=\tau_1(\Omega_t,\Omega^\prime)\,.
\end{equation}
If $t$ ranges in a bounded right neighborhood of $0$, say $[0,1]$,
then the diameters of the sets $\Omega_{i,t}$, $i\in\N$, are uniformly bounded. 
Hence, by Lemma \ref{lemma2.1} and the definition of $\tau_1$ there exists a
constant $C>0$ such that 
\begin{equation}\label{3.20}
|\tau_1(\Omega_{i,t},\Omega^\prime_i)|\le
C\quad\forall\,i\in\N\,,\,\forall\,t\in[0,1]\,.
\end{equation}
Using (\ref{3.18}), (\ref{3.19}), (\ref{3.20}) and the Dominated Convergence
Theorem we obtain
\begin{equation}\label{3.21}
\tau(\Omega_t)-\tau(\Omega)=\int_0^t\tau_1(\Omega_s,\Omega^\prime)\,ds\,,\quad t\in[0,1]\,.  
\end{equation}
Equality (\ref{3.2}) follows since the function
$s\to\tau_1(\Omega_s,\Omega^\prime)$ is continuous.
\end{proof}

By the proof of Theorem \ref{teo3.1}, and in particular from (\ref{3.21}), we
obtain the following result.

\begin{cor}\label{cor3.1} Let $\Omega$ and $\Omega^\prime$ be open bounded convex subsets of
  $\R^n$. The function $t\to\tau(\Omega+t\Omega^\prime)$, defined for $t\ge0$,
  is differentiable for every $t>0$ and
  \begin{equation}
    \label{3.22}
\frac{d}{dt}\tau(\Omega+t\Omega^\prime)=\tau_1(\Omega+t\Omega^\prime,\Omega^\prime)\,,\quad\forall\,t>0\,.     
  \end{equation}  
\end{cor}

\begin{remark}\label{remark3.3}
{\em Let $\Omega$ and $\Omega^\prime$ be as above. For
    $s\in[0,1]$ we consider the function
$$
m(s)=\tau(s\Omega+(1-s)\Omega^\prime)\,.
$$
By homogeneity, for $s\in(0,1)$ we can write
$$
m(s)=s^{n+2}\tau\left(\Omega+\frac{1-s}{s}\Omega^\prime\right)\,.
$$  
Using Corollary \ref{cor3.1} and Remark \ref{remark3.2} is quite simple to
deduce the following equalities: 
\begin{eqnarray}\label{3.23}
m^\prime(s)&=&\tau_1\left(s\Omega+(1-s)\Omega^\prime,\Omega
+\frac{1-s}{s}\Omega^\prime\right)-
\tau_1\left(s\Omega+(1-s)\Omega^\prime,\frac{1}{s}\Omega^\prime\right)\nonumber\\
&=&\int_{\partial\Omega_s}(h(\nu_s(x))-h^\prime(\nu_s(x)))|\nabla u_s(x)|^2\,d\mh(x) 
\,,
\end{eqnarray}
where: $\Omega_s=s\Omega+(1-s)\Omega^\prime$, $u_s$ is the solution of problem
(\ref{1.2}) in $\Omega_s$, $h$ and $h^\prime$ are the
support functions of $\Omega$ and $\Omega^\prime$ respectively, and $\nu_s$ is
the Gauss map of $\cl(\Omega_s)$.}
\end{remark}

\section{An extension of the Hadamard formula}\label{sec4} 

Let $g\,:\,\sfe\to\R$ be continuous and positive; we consider the set
$$
B[g]=\{x\in\R^n\,:\,\langle x,X\rangle\le g(X)\,,\;\forall\,X\in\sfe\}\,.
$$
$B[g]$ is a compact convex set, i.e. a convex body, and the origin is an
interior point of $B[g]$. Moreover,
$$
h_{B[g]}(X)\le g(X)\quad\forall\,X\in\sfe\,.
$$ 
If $K$ is a convex body such that $0\in\interno(K)$ and $f\in
C(\sfe)$, then for $t\in\R$ and $|t|$ sufficiently small
$h_K(X)+tf(X)>0$ for every $X\in\sfe$,
so that $B[h_K+tf]$ is well defined. The aim of the present section is to
prove the following result.

\begin{theorem}\label{teo4.1} Let $\Omega$ be an open bounded convex subset of
  $\R^n$ containing the origin and let $f\in C(\sfe)$. Let $h$ and $\nu$ be the support
  function and the Gauss map of $\cl(\Omega)$ respectively. Then
\begin{equation}\label{4.1}
\left.\frac{d}{dt}\tau(\interno(B[h+tf]))\right|_{t=0}=    
\int_{\partial\Omega}f(\nu(x))|\nabla u(x)|^2\,d\mh(x)\,.
\end{equation}
\end{theorem}

\begin{cor}\label{cor4.0} In the assumptions and notations of Theorem
  \ref{teo4.1} we have
\begin{equation}\label{4.1bis}
\left.\frac{d}{dt}\tau(\interno(B[h+tf]))\right|_{t=0}=    
\int_{\sfe}f(X)\,d\mu_{\tau,\Omega}(X)\,.
\end{equation}
\end{cor}

For the proof of this result we follow the argument presented in
\cite{Jerison3} and we use two lemmas proved therein, that we quote without
proof. 

\begin{lemma}\label{lemma4.1} Let $\Omega$ and $f$ be as in Theorem
  \ref{teo4.1} and let $r,R>0$ be such that 
$$
B_r\subset\Omega\subset B_R\,,
$$
where $B_r$ and $B_R$ are the balls centered at the origin with radii $r$ and
$R$ respectively. Let $h_t$ be the support function of $B[h+tf]$ (for
  $|t|$ sufficiently small) and
$$
M=\max_{\sfe}|f|\,.
$$
Then
\begin{equation}\label{4.2}
\sup_{0\le
  t\le\frac{r}{2M}}\frac{1}{t}|h_t(X)-h(X)|\le\frac{RM}{r}\quad\forall\, x\in\sfe\,.  
\end{equation}
\end{lemma}
  
\begin{lemma}\label{lemma4.2} In the notations of Lemma \ref{lemma4.1}, let
  $x\in\partial\Omega$ be such that $\nu$ is continuous at $x$ and
  let $X=\nu(x)$. Let $X_t$ be a set of points on $\sfe$ such that
  $X_t\to X$ as $t\to0$. Then
\begin{equation}\label{4.3}
\lim_{t\to0}\frac{h_t(X_t)-h(X_t)}{t}=f(X)\,.    
\end{equation}
\end{lemma}

\begin{proof}[Proof of Theorem \ref{teo4.1}.] We use the notation of the above
  lemmas; in particular, $h_t$ denotes the support function of $B[h+tf]$. For
  $s\in[0,1]$ consider the set
$$
\Omega_{s,t}=s\,\interno(B[h+tf])+(1-s)\Omega\,.
$$
$\Omega_{s,t}$ is an open bounded convex set; we denote by $h_{s,t}$ and $\nu_{s,t}$ the
support function and the Gauss map of $\cl(\Omega_{s,t})$ respectively, and by
$u_{s,t}$ the solution of (\ref{1.2}) in $\Omega_{s,t}$. By Remark \ref{remark3.3}
we have, for every $t$, 
$$
\tau(\Omega_{1,t})-\tau(\Omega_{0,t})=
\int_0^1\int_{\partial\Omega_{s,t}}|\nabla u_{s,t}(x)|^2
(h_t(\nu_{s,t}(x))-h(\nu_{s,t}(x)))\,d\mh(x)\,ds\,.
$$
On the other hand
$$
\Omega_{1,t}=\interno(B[h+tf])\,,\quad\mbox{and}\quad
\Omega_{0,t}=\Omega\,,
$$
so that
\begin{equation}\label{4.4}
\frac{\tau(\interno(B[h+tf]))-\tau(\Omega)}{t}=
\int_0^1\int_{\partial\Omega_{s,t}}|\nabla u_{s,t}(x)|^2
\frac{h_t(\nu_{s,t}(x))-h(\nu_{s,t}(x))}{t}\,d\mh(x)\,ds\,.  
\end{equation}
In the rest of the proof we will compute the limit, as $t\to0$, of the right
hand--side of (\ref{4.4}). Note that there are constant $r_1,R_1>0$, independent
of $s$ and $t$, such that
\begin{equation}\label{4.5}
  B_{r_1}\subset\Omega_{s,t}\subset B_{R_1}\,,
\end{equation}
for every $s\in[0,1]$ and $t$, where $B_{r_1}$ and $B_{R_1}$ are balls centered at the
origin with radii $r_1$ and $R_1$ respectively. Let $\rho_{s,t}$ and $r_{s,t}$ be
the radial map and the radial function of $\cl(\Omega_{s,t})$ respectively and
let 
$$
R_{s,t}=\frac{\rho^n_{s,t}(\theta)}{h_{s,t}(\nu_{s,t}(\rho_{s,t}(\theta)))}\,,
\quad\theta\in\sfe\,,
$$
be the Jacobian of the change of variable given by the radial map (see
(\ref{3.3}) in \S \ref{sec3}). By (\ref{4.5}) and Lemma \ref{lemma2.1}
there is a constant $C>0$, independent of $s$ and 
$t$, such that
\begin{equation}\label{4.5b}
\frac{1}{C}\le R_{s,t}\le C\quad\mbox{$\mh$--a.e. on $\sfe$ and}
\quad|\nabla u_{s,t}|\le C\quad\mbox{$\mh$--a.e. on $\partial\Omega_{s,t}$.}
\end{equation}
Note that for every $s\in[0,1]$, $\cl(\Omega_{s,t})$ converges to
$\cl(\Omega)$ as $t\to0$. Hence we may apply Theorem \ref{teo3.2} choosing
$h^\prime\equiv1$ (i.e. $\Omega^\prime=$ unit ball) and get
$$
\lim_{t\to0}\int_{\partial\Omega_{s,t}}|\nabla u_{s,t}|^2\,d\mh=
\int_{\partial\Omega}|\nabla u|^2\,d\mh\,,
$$
or, equivalently,
\begin{equation}\label{4.6}
\lim_{t\to0}\int_{\sfe}|\nabla u_{s,t}(r_{s,t}(\theta))|^2R_{s,t}(\theta)\,d\mh(\theta)=
\int_{\sfe}|\nabla u(r(\theta))|^2\,R(\theta)d\mh(\theta)\,,  
\end{equation}
where $r$ is the radial map of $\Omega$ and $R$ is the corresponding
Jacobian. Let
$$
g_{s,t}(\theta)=\frac{
h_t(\nu_{s,t}(r_{s,t}(\theta)))-
h(\nu_{s,t}(r_{s,t}(\theta)))}{t}\,,
\quad\theta\in\sfe\,.
$$
By Lemma \ref{lemma4.1} there exist a
constant $C_1$ independent of $s\in[0,1]$ and $t$ such that
\begin{equation}\label{4.7}
|g_{s,t}(\theta)|\le C_1\quad\forall\theta\in\sfe\,,  
\end{equation}
and, by Lemma \ref{lemma4.2},
\begin{equation}\label{4.8}
\lim_{t\to0}g_{s,t}(\theta)=f(\nu(r(\theta)))\,,\quad\mbox{for $\mh$--a.e. $\theta\in\sfe$.}  
\end{equation}
For every $s\in[0,1]$ we have
\begin{eqnarray*}
\int_{\partial\Omega_{s,t}}|\nabla
u_{s,t}(x)|^2\frac{h_t(\nu_{s,t}(x))-h(\nu_{s,t}(x))}{t}\,d\mh(x)=\\
\int_{\sfe}g_{s,t}(\theta)
|\nabla u_{s,t}(r_{s,t}(\theta))|^2R_{s,t}(\theta)\,d\mh(\theta)\,.
\end{eqnarray*}
From (\ref{4.5b})--(\ref{4.8}) and H\"older inequality it is not hard to deduce
\begin{eqnarray*}
\lim_{t\to0}\int_{\sfe}g_{s,t}(\theta)
|\nabla u_{s,t}(r_{s,t}(\theta))|^2R_{s,t}(\theta)\,d\mh(\theta)&=& 
\int_{\sfe}f(\nu(r(\theta)))|\nabla u(r(\theta))|^2\,R(\theta)d\mh(\theta)\\
&=&\int_{\partial\Omega}f(\nu(x))|\nabla u(x)|^2\,d\mh(x)\,.
\end{eqnarray*}
On the other hand, (\ref{4.5b}) and (\ref{4.7}) ensure that
$$
\left|\int_{\sfe}g_{s,t}(\theta)
|\nabla u_{s,t}(r_{s,t}(\theta))|^2R_{s,t}(\theta)\right|\le C_2\,,
$$
for some constant $C_2>0$ independent of $s$ and $t$. Consequently we may apply
the Dominated Convergence Theorem to the right hand--side of (\ref{4.4}) and
we obtain that
\begin{eqnarray*}
\lim_{t\to0}\frac{\tau(\interno(B[h+tf]))-\tau(\Omega)}{t}&=&
\int_0^1\int_{\partial\Omega}f(\nu(x))|\nabla u(x)|^2\,d\mh(x)\\
&=&\int_{\partial\Omega}f(\nu(x))|\nabla u(x)|^2\,d\mh(x)\,.
\end{eqnarray*}
\end{proof}

\section{Proof of Theorem \ref{teo1.1}}

Let $\mu$ be a non--negative Borel measure on $\sfe$ such that the assumptions
of Theorem \ref{teo1.1} are fulfilled. In particular, there exists a convex
body $\tilde K$ in $\R^n$, uniquely determined up to translations, with
non--empty interior, such that $\mu$ is the 
{\em area measure} $\sigma_{\tilde K}$ of $K$; in other words $\tilde K$ is
the solution of the classical Minkowski problem for $\mu$. We refer the reader
to \cite[Chapter 4]{Schneider} for the definition of area measure and to
\cite[Theorem 7.1.2]{Schneider} for the existence of $\tilde K$. For
simplicity, we may assume that 
$0\in\interno(\tilde K)$. Let $r>0$ be such that $\tilde K$ contains $B_r$, the ball
centered at $0$ with radius $r$. Note that $r$ depends on $\mu$ only. Throughout
this section $C$ will denote a generic positive constant depending on $\mu$,
the dimension $n$ and the constant $M$ appearing later on in (\ref{5.3}). For a fixed $X\in\sfe$ 
$$
\frac{1}{2}\int_{\sfe}|\langle
X,Y\rangle|\,d\mu(Y)=\frac{1}{2}\int_{\sfe}|\langle X,Y\rangle|
\,d\sigma_{\tilde K}(Y)
$$
is the $(n-1)$--dimensional Hausdorff measure of the orthogonal projection of
$\tilde K$ onto $X^\perp$. As $\tilde K\supset B_r$, we have 
\begin{equation}\label{5.1}
\int_{\sfe}|\langle X,Y\rangle|\,d\mu(Y)\ge C\quad\forall\,X\in\sfe\,.
\end{equation}
Now let $\Omega$ be an open bounded convex subset of $\R^n$ and let $h$ be the
support function of $K=\cl(\Omega)$. Assume that
\begin{itemize}
\item[{\em i}.]
\begin{equation}\label{5.2}
\tau(\Omega)\ge 1\,;
\end{equation}
\item[{\em ii}.]
\begin{equation}\label{5.3}
\int_{\sfe}h(X)\,d\mu(X)\le M\,,    
\end{equation}
for some positive constant $M$.
\end{itemize}
Condition {\em i}. and the monotonicity of the torsional rigidity imply
\begin{equation}\label{5.4}
\diam(\Omega)\ge C\,.  
\end{equation}
Let $d=\diam(\Omega)$; we prove that
$d\le C$ using the same argument appearing in \cite[p. 273]{Jerison2}, that we quote for
completeness. Without loss of generality we may assume that for some
$\bar X\in\sfe$, $\pm\dfrac{d}{2}\bar X\in\partial\Omega$. Then for every $X\in\sfe$ we
have
$$
h(X)=\sup_{Y\in\sfe}\langle Y,X\rangle\ge\max\left\{\pm\frac{d}{2}\langle
  X,\bar X\rangle\right\}=\frac{d}{2}|\langle X,\bar X\rangle|\,.
$$
Hence
\begin{equation}\label{5.5b}
M\ge\int_{\sfe}h(X)\,d\mu(X)\ge\frac{d}{2}\int_{\sfe}|\langle X,\bar
X\rangle|\,d\mu(X)\ge C\,.
\end{equation}
From (\ref{5.1}) we obtain
\begin{equation}\label{5.6}
\diam(\Omega)\le C\,.  
\end{equation}
Let $\cal E$ be the so--called {\em John ellipsoid} of $\Omega$, i.e. $\cal E$
is the ellipsoid of minimal volume containing $\cl(\Omega)$. We have that (see
e.g. \cite[Lemma 5.3]{Jerison2}): 
\begin{equation}\label{5.6b}
C{\cal E}\subset\cl(\Omega)\subset{\cal E}\,.
\end{equation}
Let $0<a_1\le a_2\le\dots\le a_n$ be the lengths of the semi--major axes of
$\cal E$. By (\ref{5.6})
\begin{equation}\label{5.7}
a_n\le C\,.  
\end{equation}
Let $u$ be the solution of problem (\ref{1.2}) in $\Omega$; by (\ref{5.6}) and
Lemma \ref{lemma2.1}, $|\nabla u(x)|\le C$ for every $x\in\Omega$. Hence
$$
1\le\int_\Omega|\nabla u(x)|^2\,dx\le C\ml({\cal E})\le C a_1\,,
$$
where we have used (\ref{5.7}). This implies that $a_1\ge C$. By the inclusion
(\ref{5.6b}), we have that $\Omega$ contains a ball of radius depending
only on $n$, $\mu$ and $M$. We collect the facts proved so far in this section
in the following statement.

\begin{prop}\label{prop5.1} Let $\mu$ be as in Theorem \ref{teo1.1} and let
  $\Omega$ be an open bounded subset of $\R^n$ which verifies conditions 
  i. and ii. Then there exist $x_0\in\R^n$ and two positive numbers $r,R$
  depending only on $n$, $\mu$ and $M$ such that
\begin{equation}\label{5.8}
B(x_0,r)\subset\Omega\subset B(x_0,R)\,,
\end{equation}
where $B(x_0,r)$ and $B(x_0,R)$ denote the balls centered at $x_0$ with radius
$r$ and $R$ respectively.  
\end{prop}

\begin{proof}[Proof of Theorem \ref{teo1.1}] Let us consider the following variational problem
\begin{equation}\label{5.9}
m=\inf\left\{\int_{\sfe}h_L(X)\,d\mu(X)\,|\,
\mbox{$L$ is a convex body, $\interno(L)\ne\emptyset$, $\tau(\interno(L))\ge 1$}\right\}\,. 
\end{equation}
Let $L_i$, $i\in\N$, be a minimizing sequence and let
$\Omega_i=\interno(L_i)$. Each $\Omega_i$ verifies (\ref{5.8}) for some
$x_0=x_0^i$; on the other hand, by (\ref{1.1}) the integral appearing in
(\ref{5.9}) is translation invariant and we may assume that  $x_0^i=0$ for
every $i\in\N$. By the Blaschke selection Theorem, a subsequence of $L_i$
converges to a convex body $K$ in the Hausdorff metric. For simplicity we
assume that the sequence $L_i$ itself converges to $K$. By (\ref{5.8}), $K$
has non--empty interior; let $\Omega=\interno(K)$. Using the 
continuity of $\tau$ and the uniform convergence of $h_{L_i}$ to $h_K$ as
$i\to\infty$ we obtain 
\begin{equation*}
\tau(\Omega)\ge1\quad\mbox{and}\quad
\int_{\sfe}h_K(X)\,d\mu(X)=m>0\,,
\end{equation*}
i.e. $K$ is a minimizer. Notice that if $\tau(\Omega)>1$, then we could choose
a suitable $s<1$ such that 
$$
\tau(s\Omega)\ge 1\quad\mbox{and}\quad\int_{\sfe}h_{K_s}(X)\,d\mu(X)
=s\int_{\sfe}h_K(X)\,d\mu(X)<m
$$
(with $K_s=\cl(s\Omega)$), i.e. a contradiction. Hence
$\tau(\Omega)=1$. Consider $f\in C(\sfe)$ and let $\bar t>0$ be such that
$h_t=h+tf>0$ on $\sfe$, for every $t\in(-\bar t,\bar t)$ (here, as usual, $h$
is the support function of $K$). Let
$$
K_t=\frac{1}{[\tau(\interno(B[h_t])]^{1/(n+2)}}
B[h_t]\,,\quad \Omega_t=\interno(K_t)\,,\quad t\in(-\bar t,\bar t)
$$
(see \S \ref{sec4} for the definition of $B[h_t]$). Note that $\Omega_0=\Omega$
and $\tau(\Omega_t)=1$  for $t\in(-\bar t,\bar t)$, hence the function 
$$
\phi(t)=\int_{\sfe}h_{K_t}(X)\,d\mu(X)=
\frac{1}{[\tau(\Omega_t)]^{1/(n+2)}}\int_{\sfe}h_{B[h_t]}(X)\,d\mu(X)\,,\quad
t\in(-\bar t,\bar t)\,,  
$$
has a minimum for $t=0$. 
An easy consequence of Lemmas \ref{lemma4.1} and \ref{lemma4.2} is
\begin{equation}\label{5.10}
\left.\frac{d}{dt}\int_{\sfe}h_{B[h_t]}(X)\,d\mu(X)\right|_{t=0}=
\int_{\sfe}f(X)\,d\mu(X)\,.
\end{equation}
Hence, by Theorem \ref{teo4.1}, $\phi$ is differentiable at $t=0$ and we have
$\phi^\prime(0)=0$. Applying Corollary 
\ref{cor4.0} and (\ref{5.10}) we deduce
\begin{equation*}
\int_{\sfe}f(X)\,d\mu_{\tau,\Omega}(X)=
m\,\int_{\sfe}f(X)\,d\mu(X)\,,  
\end{equation*}
where $\mu_{\tau,\Omega}$ is the measure associated to $\Omega$ as in Definition
\ref{def2.1}. As $f\in C(\sfe)$ is arbitrary, we have
$\mu_{\tau,\Omega}=m\,\mu$ and, as $\mu_{\tau,\Omega}$ is positively
homogeneous of order $(n+1)$ with respect to dilations of $\Omega$, 
$$
\tilde\Omega=\frac{1}{m^{1/(n+1)}}\,\Omega
$$
verifies $\mu_{\tau,\tilde\Omega}=\mu$. This concludes the proof of the theorem. 
\end{proof}

\bigskip

\noindent
{\sc Andrea Colesanti}\\
Dipartimento di Matematica 'U. Dini' -- Universit\`a di Firenze\\
Viale Morgagni 67/a\\
50134 Firenze, Italy\\
{\tt colesant@math.unifi.it}

\medskip

\noindent
{\sc Michele Fimiani}\\
Dipartimento di Matematica 'U. Dini' -- Universit\`a di Firenze\\
Viale Morgagni 67/a\\
50134 Firenze, Italy\\

\end{document}